\documentclass{amsart}
\usepackage{amsfonts, amsbsy, amsmath, amssymb}

\makeatletter
\@namedef{subjclassname@2010}{%
  \textup{2020} Mathematics Subject Classification}
\makeatother

\ifx\fiverm\undefined 
  \newfont\fiverm{cmr5} 
\fi

\newtheorem{thm}{Theorem}[section]
\newtheorem{lem}[thm]{Lemma}

\newtheorem{thm-con}[thm]{Theorem-Conjecture}
\numberwithin{equation}{section}

\theoremstyle{definition}

\allowdisplaybreaks

\newcommand{\f}{\Bbb F}
\newcommand{\pf}{\Bbb P^1(\f_q)}

\begin{document}

\title[Rational functions of Degree Four]{Rational functions of Degree Four that Permute the Projective Line over a Finite Field}

\author[Xiang-dong Hou]{Xiang-dong Hou}
\address{Department of Mathematics and Statistics,
University of South Florida, Tampa, FL 33620}
\email{xhou@usf.edu}



\begin{abstract}
Recently, rational functions of degree three that permute the projective line $\pf$ over a finite field $\f_q$ were determined by Ferraguti and Micheli. In the present paper, using a different method, we determine all rational functions of degree four that permute the $\pf$.
\end{abstract}

\maketitle

\section{Introduction}

Let $\Bbb P^1(\f_q)=\f_q\cup\{\infty\}$ be the projective line over the finite field $\f_q$. A rational function in $\f_q(X)$ defines a mapping from $\Bbb P^1(\f_q)$ to itself. For $0\ne f(X)=P(X)/Q(X)\in\f_q(X)$, where $P,Q\in\f_q[X]$ and $\text{gcd}(P,Q)=1$, we define $\deg f=\max\{\deg P,\deg Q\}$. Let $G(q)=\{\phi\in\f_q(X):\deg\phi=1\}$. Then $(G(q),\circ)$ is a group which is isomorphic to $\text{PGL}(2,\f_q)$. Elements of $G(q)$ induce permutations of $\Bbb P^1(\f_q)$.  Two rational functions $f,g\in\f_q(X)$ are called {\em equivalent} if there exist $\phi,\psi\in G(q)$ such that $f=\phi\circ g\circ \psi$. A polynomial $f(X)\in\f_q[X]$ that permutes $\f_q$ is called a {\em permutation polynomial} (PP) of $\f_q$; a rational function $f(X)\in\f_q(X)$ that permutes $\Bbb P^1(\f_q)$ is called a {\em permutation rational function} (PR) of $\Bbb P^1(\f_q)$. Permutation polynomials have been studied extensively. In particular, PPs of degree $\le 7$, including degree 8 in characteristic 2, have been classified \cite{Dickson-AM-1896, Fan-arXiv1812.02080, Fan-arXiv1903.10309, Li-Chandler-Xiang-FFA-2010, Shallue-Wanless-FFA-2013}. However, little is known about PRs of low degree of $\pf$. Recently, Ferraguti and Micheli \cite{Ferraguti-Micheli-arXiv1805.03097} classified PRs of degree $3$ of $\pf$ using the Chebotarev theorem for function fields. In \cite{Hou-ppt}, we showed that the same result can be obtained from a formula by Carlitz. Moreover, also using Carlitz's formula, we were able to determine all PRs of $\pf$ degree 4 except for a case where a certain condition is needed. In the present paper, we determine all degree 4 PRs in this case without any condition; therefore we now have completely determined degree 4 PRs of $\pf$. The approach of the present paper, different from that of \cite{Hou-ppt}, is based on the Hasse-Weil bound for the number of zeros of absolute irreducible polynomials over finite fields. Notably, when the aforementioned condition is removed, two new infinite classes of degree 4 PRs arise; see Theorems~\ref{T2.1} and \ref{T3.1}.

In Section~2 we recall the previous results on degree 4 PRs from \cite{Hou-ppt}. The unsettled case in \cite{Hou-ppt} is covered by Theorems~\ref{T2.1} and \ref{T3.1} of the present paper when $q$ is not small; these two theorems are proved in Sections~3 and 4, respectively. Sporadic PRs with small values of $q$ that are not covered by Theorems~\ref{T2.1} and \ref{T3.1} are given in Section~5.

\section{Previous Results}

If $f(X)$ is a PR of $\pf$ of degree $4$ which is not equivalent to a polynomial, then $f(X)$ is equivalent to $P(X)/Q(X)$, where $P,Q\in\f_q[X]$, $\text{gcd}(P,Q)=1$, $\deg Q<\deg P=4$ and $Q$ has no root in $\f_q$.

Let $f(X)=P(X)/Q(X)$, where $P,Q\in\f_q[X]$, $\text{gcd}(P,Q)=1$, $\deg Q<\deg P=4$ and $Q$ has no root in $\f_q$. If $\deg Q=2$, all such PRs of $\pf$ have been determined in \cite{Hou-ppt}; they only exist for  $q\le 8$.  Now assume that $\deg Q=3$. We can write 
\begin{equation}\label{Eq-2.1}
f(X)=X+ \frac{b}{X-r}+\frac{b^q}{X-r^q}+\frac{b^{q^2}}{X-r^{q^2}},
\end{equation}
where $b\in\f_{q^3}^*$ and $r\in\f_{q^3}$. Under the condition that $b\in\f_q^*$, all such PRs have been determined. They all occur in characteristic 3: when $q=3$, $f$ is equivalent to $X+(X^3-X+1)^{-1}$ or $X-(X^3-X+1)^{-1}$; when $q=3^n$, $n>1$, $f$ is equivalent to $X+(X^3-X+b)^{-1}$, where $b\in\f_q$ is a fixed element such that $\text{Tr}_{q/3}(b)\ne 0$. These results were obtained using a formula by Carlitz on power sums of reciprocals of polynomials. When the condition that $b\in\f_q^*$ is removed from \eqref{Eq-2.1}, computation of power sums becomes very complicated. The conclusion in \cite{Hou-ppt} about the case $\deg Q=3$ is incomplete. In the present paper, through a different approach, we will determine all PRs of degree $4$ with $\deg Q=3$ without any condition. 

In our notation, for a polynomial $F(X_1,\dots,X_n)\in\f_a[X_1,\dots,X_n]$,
\[
V_{\f_q^n}(F)=\{(x_1,\dots,x_n)\in\f_q^n: F(x_1,\dots,x_n)=0\};
\]
for a homogenous polynomial $F(X_0,\dots,X_n)\in\f_q[X_0,\dots,X_n]$, 
\[
V_{\Bbb P^n(\f_q)}(F)=\{(x_0:\ldots :x_n)\in\Bbb P^n(\f_q): F(x_0,\dots,x_n)=0\}.
\] 
A polynomial $F(X_1,\dots,X_n)$ is called {\em cyclic} if $F(X_2,\dots,X_n,X_1)=F(X_1,\dots,X_n)$. For an expression $\alpha(r_1,\dots,r_n)$ involving values $r_1,\dots, r_n$, we define
\[
\sum_{\text{cyc}}\alpha(r_1,\dots,r_n)=\alpha(r_1,\dots,r_n)+\alpha(r_2,\dots,r_n,r_1)+\cdots+\alpha(r_n,r_1,\dots,r_{n-1}).
\]

\section{Main Theorem}

Let $f(X)=P(X)/Q(X)$, where $P,Q\in\f_q[X]$, $\deg P=4$, $\deg Q=3$, $Q$ is irreducible over $\f_q$ and $Q\nmid P$. Up to equivalence, we may write
\begin{equation}\label{2.1}
f(X)=X+\frac{aX^2+bX+c}{Q(X)},
\end{equation}
where $a,b,c\in\f_q$ are not all $0$. If $\text{char}\,\f_q\ne 3$, we may assume that $Q(X)=X^3+dX+e$; when $\text{char}\,\f_q=3$, we may assume that $Q(X)=X^3+dX+e$ or $Q(X)=X^3+dX^2+e$. When $Q(X)=X^3+dX+e$, we have the following theorem.

\begin{thm}\label{T2.1} Let
\begin{equation}\label{Eq-3.2}
f(X)=X+\frac{aX^2+bX+c}{X^3+dX+e}\in\f_q(X),
\end{equation}
where $a,b,c\in\f_q$ are not all $0$ and $X^3+dX+e$ is irreducible over $\f_q$. If $q\ge 113$, then  $f$ is a PR of $\pf$ if and only if $a = -3 d$, $b = -9 e$ and $c = d^2$. The condition is sufficient for all $q$.
\end{thm}

We first prove a lemma which is a version of the Hasse-Weil bound tailored for the circumstances in this paper.

\begin{lem}\label{L}
Let $F(X,Y,Z)\in\f_q[X,Y,Z]$ be an absolutely irreducible homogeneous polynomial of degree $d$. Let $d_1=\deg_X F$, $d_2=\deg_Y F$, $d_3=\deg_Z F$, and assume that $1\le d_1\le d_2\le d_3$. Then
\[
|V_{\Bbb P^2(F)}(F)|\ge q+1-2(d_1-1)(d_2-1)q^{1/2}-\frac 12(d_2-1)(d-1)(d-2).
\]
\end{lem}

\begin{proof}
Let $\f=\Bbb F_q(x,y)$, where $x$ is transcendental over $\f_q$ and $y$ is a root of $F(x,Y,1)$. Then $\f$ is a field of algebraic functions in one variable over $\f_q$ and the absolute irreducibility of $F$ implies that $\f_p$ is algebraically closed in $\f$. By Riemann's inequality \cite[Corollary~3.11.4]{Stichtenoth-2009}, the genus $g$ of $\f/\f_q$ satisfies $g\le (d_1-1)(d_2-1)$. Let $\mathcal A_1$ be the set of degree one places of $\f/\f_q$ and let $A_1=|\mathcal A_1|$. By the Hasse-Weil bound,
\[
|A_1-(q+1)|\le 2(d_1-1)(d_2-1)q^{1/2},
\]
whence
\begin{equation}\label{L-1}
A_1\ge q+1-2(d_1-1)(d_2-1)q^{1/2}.
\end{equation}
For each $\frak P\in\mathcal A_1$, let $z\in\{x,y,1\}$ be such that $\nu_\frak P(z)=\min\{\nu_\frak P(x),\nu_\frak P(y),\nu_\frak P(1)\}$, where $\nu_\frak P$ is the valuation at $\frak P$. Then $((x/z)(\frak P):(y/z)(\frak P):(1/z)(\frak P))\in V_{\Bbb P^2(\f_q)}(F)$. This defines a mapping
\[
\begin{array}{cccl}
\psi: &\mathcal A_1 &\longrightarrow & V_{\Bbb P^2(\f_q)}(F)\vspace{0.5em}\cr
&\frak P &\longmapsto & ((x/z)(\frak P),(y/z)(\frak P), (1/z)(\frak P)).
\end{array}
\]
($\psi$ is actually onto, though we do not need this fact here.) It is well known that if $(x_0:y_0:z_0)\in V_{\Bbb P^2(\f_q)}(F)$ is a simple point, then $|\psi^{-1}(x_0:y_0:z_0)|=1$. In general, given $(x_0:y_0:z_0)\in V_{\Bbb P^2(\f_q)}(F)$, we claim that $|\psi^{-1}(x_0:y_0:z_0)|\le d_2$. If $z_0=1$, let $\frak p$ be the place of the rational function field $\f_q(x)/\f_q$ that is the zero of $x-x_0$. Every place in $\psi^{-1}(x_0:y_0:1)$ lies above $\frak p$. On the other hand, the number of places of $\f/\f_q$ lying above $\frak p$ is at most $[\f:\f_q(x)]=[\f_q(x,y):\f_q(x)]\le d_2$. Hence $|\psi^{-1}(x_0:y_0:1)|\le d_2$. If $z_0=0$, without loss of generality, assume $x_0=1$. Let $\frak p$ be the place of the rational function field $\f_q(x)/\f_q$ that is the pole of $x$. Then every place in $\psi^{-1}(1:y_0:0)$ lies above $\frak p$. It follows that $|\psi^{-1}(1:y_0:0)|\le [\f:\f_q(x)]=[\f_q(x,y):\f_q(x)]\le d_2$. Hence the claim is proved. Let $S$ be the set of multiple points on $V_{\Bbb P^2(\f_q)}(F)$. By \cite[\S 5.4, Theorem~2]{Fulton-1989}, $|S|\le(d-1)(d-2)/2$. Thus
\begin{align}\label{L-2}
A_1\,&\le |V_{\Bbb P^2(\f_q)}(F)\setminus S|+d_2|S|=|V_{\Bbb P^2(\f_q)}(F)|+(d_2-1)|S|\\
&\le |V_{\Bbb P^2(\f_q)}(F)|+\frac 12(d_2-1)(d-1)(d-2). \nonumber
\end{align}
Combining \eqref{L-1} and \eqref{L-2} gives
\begin{align*}
|V_{\Bbb P^2(\f_q)}(F)|\,&\ge A_1-\frac 12(d_2-1)(d-1)(d-2)\cr
&\ge q+1-2(d_1-1)(d_2-1)q^{1/2}-\frac 12(d_2-1)(d-1)(d-2).
\end{align*}
\end{proof}

\begin{proof}[Proof of Theorem~\ref{T2.1}] ($\Rightarrow$) We have 
\[
\frac{f(X)-f(Y)}{X-Y}=\frac{F(X,Y)}{(X^3+dX+e)(Y^3+dY+e)},
\]
where
\begin{align}\label{2.2}
&F(X,Y)=\\
&-c d + b e + e^2 + a e X + d e X - c X^2 + e X^3 + a e Y + d e Y - 
 c X Y + a d X Y \cr
 &+ d^2 X Y - b X^2 Y + d X^3 Y - c Y^2 - b X Y^2 - 
 a X^2 Y^2 + e Y^3 + d X Y^3 + X^3 Y^3.\nonumber
\end{align}
Write
\begin{equation}\label{2.3}
F(X,Y)=G(X+Y,XY),
\end{equation}
where 
\begin{align}\label{2.4}
G(X,Y)=\,&-cd+be+e^2+(ae+de)X-cX^2+eX^3+(c+ad+d^2)Y\\
&+(-b-3e)XY+dX^2Y+(-a-2d)Y^2+Y^3.\nonumber
\end{align}

\medskip
{\bf Case 1.} Assume that $G(X,Y)$ is not absolutely irreducible. Then $G(X,Y)$ has a root $\rho(X)\in\overline\f_q[X]$ for $Y$. If $\deg\rho>1$, then $\deg_XG(X,\rho(X))=3\deg\rho$, which is a contradiction. So we have $\rho(X)=uX+v$ for some $u,v\in\overline\f_q$. Comparing the coefficients of $X$ in the equation $G(X,uX+v)=0$ gives
\begin{align}\label{2.5}
&-c d + b e + e^2 + c v + a d v + d^2 v - a v^2 - 2 d v^2 + v^3 =0,\\ \label{2.6}
&a e + d e + c u + a d u + d^2 u - b v - 3 e v - 2 a u v - 4 d u v + 
 3 u v^2 =0,\\ \label{2.7}
&-c - b u - 3 e u - a u^2 - 2 d u^2 + d v + 3 u^2 v =0,\\ \label{2.8}
&e + d u + u^3 =0.
\end{align}
Since $X^3+dX+e$ is irreducible over $\f_q$, it follows from \eqref{2.8} that $u$ is of degree $3$ over $\f_q$. From \eqref{2.7},
\[
v =\frac {c + b u + 3 e u + a u^2 + 2 d u^2}{d + 3 u^2}.
\]
Making this substitution in \eqref{2.5} and \eqref{2.6}, and using \eqref{2.8} to reduce the degree of $u$, we get
\begin{align*}
&A_0+A_1u+A_2u^2=0,\cr
&B_0+B_1u+B_2u^2=0,
\end{align*}
where
\begin{align*}
A_0=\,& c^3 - a c^2 d - c^2 d^2 + a c d^3 - b^3 e - a^2 b d e + 3 b c d e - 
 a b d^2 e - 2 a^3 e^2\cr
 & - 9 b^2 e^2 + 9 a c e^2 - 6 a^2 d e^2,\cr
A_1=\,& 3 b c^2 - b^3 d - 2 a b c d - a^2 b d^2 + 3 a^2 c e - 3 a^3 d e - 
 9 b^2 d e + 9 a c d e - 9 a^2 d^2 e,\cr
A_2=\,& 3 b^2 c - a b^2 d + a^2 c d - 3 c^2 d - a^3 d^2 - 2 b^2 d^2 + 
 5 a c d^2 - 2 a^2 d^3 + 3 a^2 b e\cr
 & + 9 b c e - 3 a b d e + 9 a^2 e^2,\cr
B_0=\,& d (b c + a^2 e),\cr
B_1=\,& -3 c^2 + b^2 d + 2 a c d + a^2 d^2 - 3 a b e,\cr
B_2=\,& -3 (b c + a^2 e).
\end{align*}
Using \eqref{2.8}, we also find that
\[
A_0+A_1u+A_2u^2=(c + b u + a u^2)(A_0'+A_1'u+ A_2'u^2 ),
 \]
where
\begin{align*}
A_0'&=c^2 - a c d - c d^2 + a d^3 + a b e + 9 a e^2,\cr
A_1'&= 2 b c - 2 b d^2 + 2 a^2 e + 6 a d e,\cr
A_2'&= b^2 - a c + a^2 d - 3 c d + 
 3 a d^2 + 9 b e,
\end{align*}
and
\[
B_0+B_1u+B_2u^2=(c + b u + a u^2) (b d - 3 a e + (-3 c - a d) u).
\]
Since $u$ is of degree $3$ over $\f_q$ and since $(a,b,c)\ne(0,0,0)$, we have $A_0'=A_1'=A_2'=0$, $b d - 3 a e=0$, and $3 c + a d=0$.
 
\medskip
{\bf Case 1.1.} Assume that $p:=\text{char}\,\f_q\ne 3$. Then $c=-ad/3$, and 
\begin{align*}
A_0'&= \frac19 a (4 a d^2 + 12 d^3 + 9 b e + 81 e^2) ,\cr
A_2'&= \frac13 (3 b^2 + 4 a^2 d + 12 a d^2 + 27 b e).
\end{align*}

If $d=0$, then $c=0$. Since $e\ne 0$, it follows from $b d - 3 a e=0$ that $a=0$. Then $A_2'=0$ gives $b=-9e$ and we are done.

Now assume that $d\ne 0$. Then $b=3ae/d$, and 
\begin{align*}
A_0'&= \frac{a (a + 3 d) (4 d^3 + 27 e^2)}{9 d} ,\cr
A_2'&= \frac{a (a + 3 d) (4 d^3 + 27 e^2)}{3 d^2}.
\end{align*}
Note that $4 d^3 + 27 e^2\ne 0$ since $-4 d^3 - 27 e^2$ is the discriminant of $X^3+dX+e$. So $a (a + 3 d)=0$. If $a=0$, then $b=c=0$, which is a contradiction. If $a=-3d$, then $b=-9e$ and $c=d^2$. We are done.

\medskip
{\bf Case 1.2.} Assume that $p=3$. Since $X^3+dX+e$ is irreducible over $\f_q$, we have $d\ne 0$. It follows from $b d - 3 a e=0$, and $3 c + a d=0$ that $a=b=0$, whence $c\ne 0$. We have $A'_0=c(c-d^2)$, and hence $c=d^2$. 

\medskip
{\bf Case 2.} Assume that $G(X,Y)$ is absolutely irreducible.  

\medskip
{\bf Case 2.1.} Assume that $F(X,Y)$ is absolutely irreducible. The number of zeros of $F(X,Y)$ at infinity is $2$. Let $\overline F(X,Y,Z)$ be the homogenization of $F(X,Y)$. By Lemma~\ref{L}, 
\begin{align*}
|V_{\f_q^2}(F)|+2\,&=|V_{\Bbb P^2(\f_q)}(\overline F)|\ge q+1-2(3-1)(3-1)q^{1/2}-\frac 12(3-1)(6-1)(6-2)\cr
&=q-8q^{1/2}-19.
\end{align*}
On the other hand, $|V_{\f_q}(F(X,X))|\le\deg F(X,X)=6$. Thus
\[
|V_{\f_q^2}(F(X,Y))|\ge q-8q^{1/2}-21>6=|V_{\f_q}(F(X,X))|.
\]
Hence there exists $(x,y)\in\f_q^2$ with $x\ne y$ such that $F(x,y)=0$, which is a contradiction.

\medskip
{\bf Case 2.2.} Assume that $F(X,Y)$ is not absolutely irreducible. Then $F(X,Y)$ has a factor $A (Y) X + 
  B (Y)$ where $A(Y),B(Y)\in \overline\f_q[Y]$, $A(Y)\ne 0$ and $\text{gcd}(A(Y),B(Y))=1$. Since $G(X,Y)$ is absolutely irreducible, $A(Y) X + B (Y)$ cannot be symmetric in $X$ and $Y$. Since $F(X,Y)$ is symmetric, $A (X) Y + B (X)$ is also a factor of $F(X,Y)$. Note that $\text{gcd}(A (Y) X + B (Y), A (X) Y + B (X))=1$. (Otherwise, $A (Y) X + B (Y)=\epsilon(A (X) Y + B (X))$ for some $\epsilon\in\overline\f_q^*$. It follows easily that $A (X) Y + B (X)=\delta(X-Y)$ for some $\delta\in\overline\f_q^*$. This is a contradiction since $X-Y\nmid F(X,Y)$.) Thus $(A (Y) X + B (Y)) (A (X) Y + B (X))$ is a symmetric factor of $F(X,Y)$. By the absolute irreducibility of $G(X,Y)$, we have $F(X,Y)=\eta(A (Y) X + B (Y)) (A (X) Y + B (X))$ for some $\eta\in\overline\f_q^*$. We may assume that one of the nonzero coefficients of $A (Y) X + B (Y)$ is in $\f_q$. We then claim that $A(Y)X+B(Y)\in\f_q[X,Y]$. (Otherwise, there exits $\sigma\in\text{Aut}(\overline \f_q/\f_q)$ such that $\sigma(A(Y)X+B(Y))\ne A(Y)X+B(Y)$. Note that $\sigma(A(Y)X+B(Y))$ is a factor of $F(X,Y)$ which is coprime to $A(Y)X+B(Y)$. Hence $\sigma(A(Y)X+B(Y))\mid A(X)Y+B(X)$, which is impossible since $\deg_Y(A(Y)X+B(Y))=2$ and $\deg_Y(A(X)Y+B(X))=1$.)  Therefore,
\[
|V_{\f_q^2}(F(X,Y))|\ge|V_{\f_q^2}(A(Y)X+B(Y))|\ge q-\deg A\ge q-2>6=|V_{\f_q}(F(X,X))|,
\]
which is a contradiction.

\medskip
($\Leftarrow$) Let $u\in\f_{q^3}$ be a root of $X^3+dX+e$ and let $u_1=u$, $u_2=u^q$, $u_3=u^{q^2}$. Then
\begin{align*}
&u_1+u_2+u_3=0\cr
&u_1u_2+u_2u_3+u_3u_1=d,\cr
&u_1u_2u_3=-e.
\end{align*} 
Using these equations, together with the conditions  $a = -3 d$, $b = -9 e$ and $c = d^2$, one easily verifies that the polynomial in $G(X,Y)$ in \eqref{2.4} affords the factorization
\[
G(X,Y)=\prod_{i=1}^3(Y-u_iX-d-2u_i^2).
\]
Therefore, for the polynomial $F(X,Y)$ in \eqref{2.2}, we have
\[
F(X,Y)=\prod_{i=1}^3(XY-u_i(X+Y)-d-2u_i^2).
\]
If $p\ne 2$, $XY-u_i(X+Y)-d-2u_i^2$ has no root $(x,y)\in\f_q^2$; if $p=2$, $XY-u_i(X+Y)-d-2u_i^2$ has no root $(x,y)\in\f_q^2$ with $x\ne y$. In both cases, $f(X)$ is a PR of $\pf$.
\end{proof}

\section{The Other Case in Characteristic $3$}

Assume that $\text{char}\,\f_q=3$. The rational functions $f$ of the form \eqref{2.1}, where $Q(X)=X^3+dX^2+e$, are not covered by Theorem~\ref{T2.1}. Functions of this type need to be treated separately although the argument is similar. First notice that by replacing $X$ be $dX$, we may assume, up to equivalence, that $d=1$, i.e., $f(X)=X+(a X^2 + b X + c)/(X^3 +  X^2 + e)\in\f_q(X)$.

\begin{thm}\label{T3.1}
Let $q=3^n$ and let
\begin{equation}\label{3.1}
f(X)=X + \frac{a X^2 + b X + c}{X^3 +  X^2 + e}\in\f_q(X),
\end{equation}
where $a,b,c\in\f_q$ are not all $0$ and $X^3 +  X^2 + e$ is irreducible over $\f_q$. If $n>4$, then $f$ is a PR of $\pf$ if and only if $a=1$ and $b=c=0$. The condition is sufficient for $n>0$
\end{thm}

\begin{proof}
($\Rightarrow$)
We have 
\[
\frac{f(X)-f(Y)}{X-Y}=\frac{F(X,Y)}{(X^3+X^2+e)(Y^3+Y^2+e)},
\]
where
\begin{align}\label{3.2}
&F(X,Y)=\\
&b e + e^2 + (-c   + a e )X + (-c +  e )X^2 + e X^3 + (-c   + 
 a e )Y + (-c  -b  )xy\cr
 & - b X^2 Y + (-c  +  e )Y^2 -b X Y^2 + (-a  + 1 )X^2 Y^2 +  X^3 Y^2 + e Y^3 + 
  X^2 Y^3 + X^3 Y^3.\nonumber
\end{align}
Write
\begin{equation}\label{3.3}
F(X,Y)=G(X+Y,XY),
\end{equation}
where 
\begin{align}\label{3.4}
G(X,Y)=\,&b e + e^2 + (-c   + a e )X + (-c  +  e )X^2 + e X^3 + (c  -b   +  e )Y\\
 & - b X Y + (-a  + 1 )Y^2 +  X Y^2 + Y^3.\nonumber
\end{align}

\medskip
{\bf Case 1.} Assume that $G(X,Y)$ is not absolutely irreducible. Then $G(X,Y)$ has a root $\rho(X)\in\overline\f_q[X]$ for $Y$. If $\deg\rho>1$, then $\deg_XG(X,\rho(X))=3\deg\rho$, which is a contradiction. So we have $\rho(X)=uX+v$ for some $u,v\in\overline\f_q$. Comparing the coefficients of $X$ in the equation $G(X,uX+v)=0$ gives
\begin{align}\label{3.5}
&b e + e^2 + c v- b  v +  e v - a v^2 +  v^2 + v^3 =0,\\ \label{3.6}
&-c  +  a e +  c u - b  u +   e u - b v +  a u v -  u v + 
     v^2 =0,\\ \label{3.7}
&-c +   e - b u - a u^2 +  u^2 -  u v =0,\\ \label{3.8}
&e +  u^2 + u^3 =0.
\end{align}
Since $X^3+X^2+e$ is irreducible over $\f_q$, it follows from \eqref{3.8} that $u$ is of degree $3$ over $\f_q$. By \eqref{3.7},
\[
v =\frac {-c +  e - b u - a u^2 +  u^2}{ u}.
\]
Making this substitution in \eqref{3.5} and \eqref{3.6}, and using \eqref{2.8} to reduce the degree of $u$, we get
\begin{align*}
&A_0+A_1u+A_2u^2=0,\cr
&B_0+B_2u^2=0,
\end{align*}
where
\begin{align*}
A_0=\,& c^3 - b^3 e - a b^2  e + a^2 c  e - a^2 b  e - b c  e - b^2  e - a c  e + a b  e + a^3 e^2 - a^2  e^2,\cr
A_1=\,& a c^2  - c^2  + a^2 b  e - a b  e,\cr
A_2=\,& -b^3  - a b c  - a b^2  + a^2 c + c^2  - a^2 b  -
  b c  - b^2  - a c  + a b  + a^3  e + a b  e - a^2  e,\cr
B_0=\,& c^2 + a b e,\cr
B_2=\,& -b^2 + a c + a b .
\end{align*}
Using \eqref{3.8}, we also find that
\[
 B_0 + B_2 u^2=(c - b u) (c + b u + a u^2).
\]
It follows that $b=c=0$, whence $a\ne 0$. Now $A_0=a^2e^2(a-1)$, and hence $a=1$. 

\medskip
{\bf Case 2.} Assume $G(X,Y)$ is absolutely irreducible. The proof is identical to Case 2 in the proof of Theorem~\ref{T2.1}.

\medskip
($\Leftarrow$) Let $u\in\f_{q^3}$ be a root of $X^3+X^2+e$ and let $u_1=u$, $u_2=u^q$, $u_3=u^{q^2}$. Then
\begin{align*}
&u_1+u_2+u_3=-1\cr
&u_1u_2+u_2u_3+u_3u_1=0,\cr
&u_1u_2u_3=-e.
\end{align*} 
Using these equations, together with the conditions  $a = 1$ and $b =c = 0$, one easily verifies that the polynomial in $G(X,Y)$ in \eqref{3.4} affords the factorization
\[
G(X,Y)=\prod_{i=1}^3(Y-u_iX+u_i+u_i^2).
\]
Therefore, for the polynomial $F(X,Y)$ in \eqref{2.2}, we have
\[
F(X,Y)=\prod_{i=1}^3(XY-u_i(X+Y)+u_i+u_i^2),
\]
which has no root $(x,y)\in\f_q^2$. Therefore, $f(X)$ is a PR of $\pf$.
\end{proof}

\section{Cases of Small $q$}

In Theorem~\ref{T2.1}, the $q$'s that are not covered by the necessity of the condition are $2^1,\dots,2^6$, $3^1,\dots,3^4$, $5,5^2$, $7,7^2$, 11, 13, 17, 19, 23, 29, 31, 37, 41, 43, 47, 53, 59, 61, 67, 71, 73, 79, 83, 
89, 97, 101, 103, 107, 109. In Theorem~\ref{T3.1}, the $q$'s not covered by the necessity of the condition are $3^1,\dots,3^4$. We did a computer search for these $q$'s; sporadic PRs not covered by Theorems~\ref{T2.1} and \ref{T3.1} are given in Table~\ref{Tb1} and Table~\ref{Tb2}. The results of Table~\ref{Tb1} should be interpreted with the following observations.

In Theorem~\ref{T2.1}, $f(X)=X+(aX^2+bX+c)/(X^3+dX+e)$. If $q$ is odd, by a suitable substitution $X\mapsto tX$ ($t\in\f_q^*$), we may assume $a\in\{0,1,u\}$, where $u$ is a chosen nonsquare in $\f_q^*$; if $\text{char}\,\f_q=3$, by a substitution $X\mapsto vX+w$ ($v\in\f_q^*$, $w\in\f_q$), we may further assume that $(a,b)=(0,0)$ or $(a,b,c)=(0,1,0)$ or $(a,b)=(1,0)$ or $(a,b)=(u,0)$. If $q$ is even, by a substitution $X\mapsto tX$ ($t\in\f_q^*$), we may assume that $a=0$ or $1$.

\begin{table}
\renewcommand*{\arraystretch}{1.2}
\caption{PRs of the form $X+(aX^2+bX+c)/(X^3+dX+e)$ not covered by Theorem~\ref{T2.1}}
\label{Tb1}
\vskip-1.5em
\[
\begin{tabular}{c|l}
\hline
$q$ & $\{a,b,c,d,e\}$ \\ \hline
$2$ & $\{0,0,1,1,1\},\ \{1,1,0,1,1\}$ \\ \hline
$2^2$ & $\{0,1,0,0,u\},\ \{0,1,0,0,1+u\},\ \{1,0,0,1+u,1\},\ \{1,u,0,0,1+u\},$ \\ 
      & $\{1,1+u,0,0,u\}$, where $u^2+u+1=0$. \\ \hline
$3$ & $\{0, 0, 2, 2, 1\}, \ \{0, 0, 2, 2, 2\},\ \{0, 1, 0, 2, 1\}$ \\ \hline
$5$ & $ \{0, 1, 4, 4, 2\},\ \{0, 2, 4, 1, 4\},\ \{0, 3, 4, 1, 1\},\ \{0, 4, 4, 4, 3\},$\\ 
& $\{1, 2, 2, 2, 4\},\ \{1, 3, 2, 2, 1\},\ \{2, 1, 3, 4, 3\},\ \{2, 2, 0, 4, 2\},$\\ 
& $\{2, 2, 3, 2, 1\},\ \{2, 2, 4, 1, 1\},\ \{2, 3, 0, 4, 3\},\ \{2, 3, 3, 2, 4\},$\\ 
& $\{2, 3, 4, 1, 4\},\ \{2, 4, 3, 4, 2\}$\\ \hline
$7$ & $\{0, 2, 0, 0, 5\},\ \{0, 5, 0, 0, 2\},\ \{1, 0, 2, 5, 2\},\ \{1, 0, 2, 5, 5\}$ \\ \hline
\end{tabular}
\]
\end{table} 

\begin{table}
\renewcommand*{\arraystretch}{1.2}
\caption{PRs of the form $X+(aX^2+bX+c)/(X^3+X^2+e)$ not covered by Theorem~\ref{T3.1}}\label{Tb2}
\vskip-1.5em
\[
\begin{tabular}{c|l}
\hline
$q$ & $\{a,b,c,e\}$ \\ \hline
$3$ & $\{0,2,1,2\},\ \{1,1,2,2\},\ \{2,1,2,2\},\ \{2,2,1,2\}$ \\ \hline
\end{tabular}
\]
\end{table}

To reduce the amount of computation in the search, we used the following two lemmas which gives necessary conditions the $f(X)$ in Theorems~\ref{T2.1} and \ref{T3.1} to be a PR. 

\begin{lem}\label{L5.1}  
Let $f(X)=X+(aX^2+bX+c)/(X^3+dX+e)\in\f_q(X)$, where $X^3+dX+e$ is irreducible over $\f_q$, and assume that $f$ is a PR of $\pf$.
\begin{itemize}
\item[(i)] If $q>2$, then
\begin{equation}\label{5.1} 
9 c^2 d^3 + 9 b^2 d^4 + 6 a c d^4 + a^2 d^5 + 81 b c d^2 e - 27 a b d^3 e + 243 c^2 e^2 - 81 a c d e^2 + 27 a^2 d^2 e^2=0.
\end{equation}

\item[(ii)] If $q>3$, then
\begin{equation}\label{5.2}
1521c^4d^6-534b^2c^2d^7-\cdots+708588a^2e^8=0.
\end{equation}
(The entirety of equation~\eqref{5.2} is given in the appendix.)
\end{itemize}
\end{lem}

Before proving the lemma, we recall two facts:

\subsection*{Hermite criterion} A function $f:\f_q\to\f_q$ is a permutation of $\f_q$ if and only if 
\begin{equation}\label{hc}
\sum_{x\in\f_q}f(x)^s=\begin{cases}
0&\text{if}\ 1\le s\le q-2,\cr
-1&\text{if}\ s=q-1.
\end{cases}
\end{equation}

\subsection*{Reciprocal power sums {\rm\cite{Hicks-Hou-Mullen-AMM-2012, Hou-ppt}}}
We have
\begin{equation}\label{rps}
\sum_{x\in\f_q}\frac 1{(x-X)^k}=\frac 1{(X^q-X)^k},\quad 1\le k\le q.
\end{equation}

\medskip
\begin{proof}[Proof of Lemma~\ref{L5.1}]
Let $X^3+dX+e=(X-r_1)(X-r_2)(X-r_3)$, where $r_2=r_1^q$ and $r_3=r_2^q$. We have
\[
f(X)=X+\sum_{\text{cyc}}\frac{c+br_1+ar_1^2}{(r_1-r_2)(r_1-r_3)}\cdot\frac 1{X-r_1}.
\]
Therefore
\begin{align*}
0\,&=\sum_{x\in\f_q}f(x)\kern12.5em\text{(Hermite criterion)}\cr
&=\sum_{\text{cyc}}\frac{c+br_1+ar_1^2}{(r_1-r_2)(r_1-r_3)}\sum_{x\in\f_q}\frac 1{x-r_1}\cr
&=\sum_{\text{cyc}}\frac{c+br_1+ar_1^2}{(r_1-r_2)(r_1-r_3)}\cdot\frac 1{r_2-r_1}\kern3em\text{(by \eqref{rps})}\cr
&=\frac{g(a,b,c,r_1,r_2,r_3)}{(r_1-r_2)^2(r_2-r_3)^2(r_3-r_1)^2},
\end{align*}
where $g(a,b,c,r_1,r_2,r_3)$ is a polynomial in $a,b,c,r_1,r_2,r_3$ over $\f_p$ which is cyclic in $r_1,r_2,r_3$. Thus 
\[
g(a,b,c,r_1,r_2,r_3)=0.
\]
To derive an equation in terms of $a,b,c,d,e$, we consider the symmetrization of $g$:
\[
g(a,b,c,r_1,r_2,r_3)g(a,b,c,r_2,r_1,r_3).
\]
This polynomial is symmetric in $r_1,r_2,r_3$ and hence is a polynomial in $a,b,c,d,e$, which is precisely the polynomial in \eqref{5.1}.

\medskip
(ii) We have
\begin{align*}
f(X)^2=\,& X^2+2a\cr
&+\sum_{\text{cyc}}\Bigl(\frac{A(a,b,c,r_1)}{(r_1-r_2)^2(r_1-r_3)^2}\cdot\frac 1{(X-r_1)^2}+\frac{B(a,b,c,r_1,r_2,r_3)}{(r_1-r_2)^3(r_1-r_3)^3}\cdot\frac 1{X-r_1}\Bigr),
\end{align*}
where
\[
A(a,b,c,r_1)=c^2 + 2 b c r_1 + b^2 r_1^2 + 2 a c r_1^2 + 2 a b r_1^3 + a^2 r_1^4
\]
and $B(a,b,c,r_1,r_2,r_3)$ is polynomial in $a,b,c,r_1,r_2,r_3$ over $\f_p$. Therefore
\begin{align*}
0\,&=\sum_{x\in\f_q}f(x)^2\kern20.4em\text{(Hermite criterion)}\cr
&=\sum_{\text{cyc}}\Bigl(\frac{A(a,b,c,r_1)}{(r_1-r_2)^2(r_1-r_3)^2}\sum_{x\in\f_q}\frac 1{(X-r_1)^2}+\frac{B(a,b,c,r_1,r_2,r_3)}{(r_1-r_2)^3(r_1-r_3)^3}\sum_{x\in\f_q}\frac 1{X-r_1}\Bigr)\cr
&=\sum_{\text{cyc}}\Bigl(\frac{A(a,b,c,r_1)}{(r_1-r_2)^2(r_1-r_3)^2}\cdot\frac 1{(r_2-r_1)^2}+\frac{B(a,b,c,r_1,r_2,r_3)}{(r_1-r_2)^3(r_1-r_3)^3}\cdot\frac 1{r_2-r_1}\Bigr)\cr
&\kern30.5em \text{(by \eqref{rps})}\cr
&=\frac{h(a,b,c,r_1,r_2,r_3)}{(r_1-r_2)^4(r_2-r_3)^4(r_3-r_1)^4},
\end{align*}
where $h(a,b,c,r_1,r_2,r_3)$ is a polynomial in $a,b,c,r_1,r_2,r_3$ over $\f_q$ which is cyclic in $r_1,r_2,r_3$. The symmetrization of $h$,
\[
h(a,b,c,r_1,r_2,r_3)h(a,b,c,r_2,r_1,r_3),
\]
is a polynomial in $a,b,c,d,e$, which is precisely the polynomial in \eqref{5.2}.
\end{proof}

\begin{lem}\label{L5.2}  
Let $q=3^n$ and $f(X)=X+(aX^2+bX+c)/(X^3+X^2+e)\in\f_q(X)$, where $X^3+X^2+e$ is irreducible over $\f_q$. Assume that $f$ is a PR of $\pf$.
\begin{itemize}
\item[(i)] We have $c^2+b^2e=0$.

\item[(i)]
If $n>1$, then
\begin{align}\label{5.5}
&a b^6 + b^7 + a b^4 c + 2 b^5 c + 2 a b^5 c + b^6 c +
2 a^2 b^2 c^2 + a^3 b^2 c^2 + 2 a b^3 c^2\\
& + a^2 b^3 c^2 + a^2 c^3 + a^3 c^3 + a^4 c^3 + 2 b^2 c^3=0. \nonumber
\end{align}
\end{itemize}
\end{lem}

\begin{proof}
The proof follows from the same computation as in the proof of Lemma~\ref{L5.1}. (i) follows from $\sum_{x\in\f_q}f(x)=0$. For (ii), one first obtains an equation in $a,b,c,e$ from $\sum_{x\in\f_q}f(x)^2=0$. Under the condition $c^2+b^2e=0$, that equation becomes \eqref{5.5}. We omit the details of the computation.  
\end{proof}

The above lemmas are special cases of a new approach that allows us derive equations that are satisfied the coefficients of a PR.



\section*{Appendix}

The polynomial in the left side of \eqref{5.2} is
\begin{align*}
& 1521 c^4 d^6 - 534 b^2 c^2 d^7 - 1716 a c^3 d^7 + 121 b^4 d^8 + 700 a b^2 c d^8 + 406 a^2 c^2 d^8 - 1872 c^3 d^8\cr
& + 58 a^2 b^2 d^9 +
44 a^3 c d^9 + 816 b^2 c d^9 + 2928 a c^2 d^9 + a^4 d^{10} - 432 a b^2 d^{10} - 1008 a^2 c d^{10}\cr
& + 576 c^2 d^{10} - 48 a^3 d^{11} +
64 b^2 d^{11} - 1152 a c d^{11} + 576 a^2 d^{12} - 16 038 b c^3 d^5 e \cr
& + 5346 b^3 c d^6 e + 16 110 a b c^2 d^6 e -
1254 a b^3 d^7 e - 2346 a^2 b c d^7 e + 15 768 b c^2 d^7 e \cr
& - 222 a^3 b d^8 e - 792 b^3 d^8 e - 16 848 a b c d^8 e +
2232 a^2 b d^9 e - 1728 b c d^9 e + 1344 a b d^{10} e \cr
& - 15 552 c^4 d^3 e^2 + 75 897 b^2 c^2 d^4 e^2 + 29 322 a c^3 d^4 e^2 -
2376 b^4 d^5 e^2 - 35 748 a b^2 c d^5 e^2 \cr
& - 15 660 a^2 c^2 d^5 e^2 + 972 c^3 d^5 e^2 + 4761 a^2 b^2 d^6 e^2 +
2538 a^3 c d^6 e^2 - 15 228 b^2 c d^6 e^2 \cr
& + 324 a c^2 d^6 e^2 + 216 a^4 d^7 e^2 + 4644 a b^2 d^7 e^2 + 2052 a^2 c d^7 e^2 +
7776 c^2 d^7 e^2 \cr
& - 3348 a^3 d^8 e^2 + 2592 b^2 d^8 e^2 - 20 736 a c d^8 e^2 + 13 536 a^2 d^9 e^2 + 157 464 b c^3 d^2 e^3 \cr
& - 52 488 b^3 c d^3 e^3 - 143 370 a b c^2 d^3 e^3 + 15 876 a b^3 d^4 e^3 + 62 370 a^2 b c d^4 e^3 + 94 770 b c^2 d^4 e^3 \cr
& - 8100 a^3 b d^5 e^3 + 2430 b^3 d^5 e^3 - 148 716 a b c d^5 e^3 + 11 826 a^2 b d^6 e^3 - 23 328 b c d^6 e^3 \cr
& + 23 328 a b d^7 e^3 + 124 659 c^4 e^4 - 61 236 b^2 c^2 d e^4 - 122 472 a c^3 d e^4 + 11 664 b^4 d^2 e^4 \cr
& + 125 388 a b^2 c d^2 e^4 + 75 087 a^2 c^2 d^2 e^4 + 91 854 c^3 d^2 e^4 - 33 048 a^2 b^2 d^3 e^4 - 29 160 a^3 c d^3 e^4 \cr
& - 139 968 b^2 c d^3 e^4 - 218 700 a c^2 d^3 e^4 + 5103 a^4 d^4 e^4 + 62 694 a b^2 d^4 e^4 + 106 434 a^2 c d^4 e^4 \cr
& + 26 244 c^2 d^4 e^4 - 32 076 a^3 d^5 e^4 + 26 244 b^2 d^5 e^4 - 122 472 a c d^5 e^4 + 119 556 a^2 d^6 e^4 \cr 
& + 91 854 a b c^2 e^5 - 34 992 a b^3 d e^5 - 69 984 a^2 b c d e^5 - 78 732 b c^2 d e^5 + 21 870 a^3 b d^2 e^5 \cr
& + 52 488 b^3 d^2 e^5 - 236 196 a b c d^2 e^5 - 56 862 a^2 b d^3 e^5 - 78 732 b c d^3 e^5 + 131 220 a b d^4 e^5 \cr
& + 26 244 a^2 b^2 e^6 - 590 490 a c^2 e^6 + 78 732 a b^2 d e^6 + 314 928 a^2 c d e^6 - 78 732 a^3 d^2 e^6 \cr
& + 78 732 b^2 d^2 e^6 - 236 196 a c d^2 e^6 + 472 392 a^2 d^3 e^6 - 236 196 a^2 b e^7 + 236 196 a b d e^7 \cr
& + 708 588 a^2 e^8.
\end{align*}

\end{document}